\documentclass[12pt,notitlepage,a4paper]{article}
\usepackage{amsmath,amssymb,amsthm,amscd}
\usepackage[arrow,matrix]{xy}
\theoremstyle{plain}
\numberwithin{equation}{section}
\newtheorem{thm}{Theorem}[section]
\newtheorem{prop}[thm]{Proposition}

\newtheorem{slem}[thm]{Sublemma}
\newtheorem{lem}[thm]{Lemma}
\newtheorem{conj}[thm]{Conjecture}
\theoremstyle{definition}
\newtheorem{df}[thm]{Definition}
\newtheorem{ex}[thm]{Example}
\newtheorem{rmk}[thm]{Remark}
\makeatletter
 
  \@addtoreset{equation}{section}
 \makeatother
\newcommand{\mb}{\mathbb}
\newcommand{\mf}{\mathfrak}
\newcommand{\ml}{\mathcal}
\begin{document}
\author{Rin Sugiyama\footnote{rin-sugiyama@math.nagoya-u.ac.jp}\\ \textit{\small Graduate School of Mathematics, Nagoya University,}\\ \textit{\small Furo-cho, Chikusa-ku, Nagoya 464-8602, Japan}}
\title{On the kernel of the reciprocity map of simple normal crossing varieties over finite fields}
\date{}
\maketitle
\begin{abstract}

In this paper, we study the kernel of the reciprocity map of certain simple normal crossing varieties over a finite field and give a example of a simple normal crossing surface whose reciprocity map is not injective for any finite scalar extension.
\end{abstract}

%%%Introduction%%%%%%%%%%%%%%%%%%%%%%%%%%%%%%%%%%%%%%%%%%%%%%%%%%%%%%%%%%%%%%%%%%%%%%%%%%%%%%%%%%%%%%%%%%%%%%%%%%%%%%%%%%%%%%%%%%%%%%%%%%%%%%%%%%%%%
%%%%%%%%%%%%%%%%%%%%%%%%%%%%%%%%%%%%%%%%%%%%%%%%%%%%%%%%%%%%%%%%%%%%%%%%%%%%%%%%%%%%%%%%%%%%%%%%%%%%%%%%%%%%%%%%%%%%%%%%%%%%%%%%%%%%%%%%%%%%%%%%%%%%
\section{Introduction}
The reciprocity map of the unramified class field theory for a proper variety $X$ over a finite field $k$ is a homomorphism of the following form:
\begin{align*}
\rho _X : CH_0(X) \longrightarrow \pi _1^{ab}(X),
\end{align*}
which sends the class of a closed point $x$ to the Frobenius substitution at $x$. Here $CH_0(X)$ is the Chow group of 0-cycles on $X$ modulo rational equivalence, and $\pi_1^{ab}(X)$ is the abelian $\acute{\text{e}}$tale fundamental group of $X$. If $X$ is normal, $\rho _X$ has dense image \cite{L}. If $X$ is smooth, $\rho _X$ is injective \cite{KS}. We also know that there is a projective normal surface $X$ for which $\rho _X$ is not injective \cite{MSA}, and that there is a simple normal crossing surface $X$ over $k$ for which $\rho _X/n$ is not injective but $\rho _{X\otimes E}/n$ is injective for any sufficiently large finite extension $E/k$ and some $n>1$ \cite{Sat}. Here a normal crossing variety $X$ over $k$ is a separated scheme of finite type over $k$ which is everywhere $\acute{\text{e}}$tale locally isomorphic to
\begin{align*}
\mathrm{Spec}\bigl(k[T_0,\cdots, T_d]/(T_0T_1\cdots T_r)\bigr) \ \ \ (0 \leq r\leq d=\dim X).
\end{align*}
A normal crossing variety $X$ over $k$ is called simple if any irreducible component of $X$ is smooth over $k$. For any simple normal crossing variety $X$, we have an exact sequence (cf. Section 1.3)
\begin{align*}
\xymatrix{
H_2(\Gamma _X, \mb{Z}/n)\ar[r]^{\epsilon _{X,n}} &CH_0(X)/n \ar[r]^{\rho_ X/n}&\pi_1^{ab}(X)/n,
}
\end{align*}
where $\Gamma _X$ is the dual graph of $X$ which is a finite simplicial complex (cf. Section 1.1). Hence, by studying on the map $\epsilon _{X,n}$, we get an information about the injectivity of $\rho_X/n$. However $\epsilon _{X,n}$ is abstract and difficult to compute directly. In this paper, for a certain simple normal crossing variety $X$ over $k$, we investigate $\epsilon _{X,n}$ and the kernel of the reciprocity map $\rho _X$ by using the $\acute{\text{e}}$tale homology theory and the cohomological Hasse principle. We also construct a simple normal crossing surface $Y$ for which the map
\begin{align*}
\rho _{Y\otimes F}/n : CH_0(Y\otimes _kF)/n\rightarrow \pi_1^{ab}(Y\otimes _kF)/n
\end{align*}
is not injective for any finite extension $F/k$ and some $n>1$. 

Let $Y_0$ a projective smooth and geometrically irreducible variety over $k$ and $D$ be a simple normal crossing divisor on $Y_0$. We put $O:=(0:1), \infty :=(1:0) \in \mb{P}^1_k$. We then consider the following simple normal crossing variety:
\begin{align*}
Y := \bigl( Y_0\times _kO \bigr) \cup \bigl( Y_0\times _k\infty \bigr) \cup \bigl( D\times _k\mb{P}^1\bigr) \ \ \subset Y_0\times _k\mb{P}^1.
\end{align*}

The following map plays an important role on the kernel of the reciprocity map $\rho _Y$ of $Y$:
\begin{align*}
\delta _Y : H_1(\Gamma _D ,\mb{Z}) \longrightarrow CH_0(Y).
\end{align*}
We will construct the map $\delta _Y$ in Section 2.1 and prove the image coincides with the kernel of the reciprocity map $\rho _Y$. We consider the norm map $\sigma : H_1(\Gamma _{\overline{D}},\mb{Z})\rightarrow H_1(\Gamma _{D},\mb{Z})$, and put $G(Y)$ the image of the composite map $\delta _Y \circ \sigma $. The group $G(Y)$ relates to the following group and map (cf. Section 2.2)
\begin{align*}
&\Theta _{\ell }:=\mathrm{Coker} \bigl( \bigoplus _j\pi _1^{ab}(\overline{D}_j)^{pro-\ell } \longrightarrow \pi _1^{ab}(\overline{Y_0})^{pro-\ell } \bigr),\\
&\alpha ^{(\ell )} : H_1(\Gamma _{\overline {D}} ,\mb{Z}_{\ell }) \longrightarrow \Theta _{\ell }.
\end{align*}
Here $\overline{Y_0}:=Y_0\otimes _kk^{sep}$ and $\overline{D}_j$ denotes irreducible component of $\overline{D}:=D\otimes _kk^{sep}$, and $\pi_1^{ab}(-)^{pro-\ell}$ denotes the maximal $\ell $-quotient of $\pi_1^{ab}(-)$.

The main result of this paper is the following theorem:
\begin{thm}[Theorem \ref{MT}]\label{IT}
Let $\ell $ be an arbitrary prime number.\\
\textup{(1)} The $\ell $-primary part $G(Y)\{\ell \}$ of $G(Y)$ is a subquotient of $\big(\Theta _{\ell }\bigr)_{tors}$.\\
\textup{(2)} Assume that

\ \ \ \textup{(i)} each connected components of $Y^{(2)}$ has a $k$-rational point,

\ \ \ \textup{(ii)} $G_k$ acts on $\big(\Theta _{\ell }\bigr)_{tors}$ trivially. 

Then $G(Y)\{\ell \}$ is isomorphic to the image of the map $\alpha ^{(\ell )}$.
\end{thm}

In their paper \cite{MSA}, Matsumi, Sato and Asakura proved a similar assertion of Theorem \ref{IT} for projective normal surfaces over finite fields. Theorem \ref{IT} shows an analogy of their result for certain simple normal crossing varieties over finite fields.

The remarkable point in Theorem \ref{IT} is the map $\alpha ^{(\ell )}$ does not vary for finite scalar extensions and the group $\mathrm{Im}(\alpha ^{(\ell )})$ relates with $G(Y)$. Therefore we will get a information about the potential injectivity of the reciprocity map $\rho _Y$ by studying on $G(Y)$ and using the map $\alpha ^{(\ell )}$. If the assumptions of Theorem \ref{IT}(2) is satisfied and $\mathrm{Im}(\alpha ^{(\ell) })$ is not trivial, then for any finite extension $F/k$ the map $\rho _{Y\otimes F}$ is not injective, i.e. potentially not injective. We will give such a surface in Section 3.1.\bigskip

This paper is organized as follows : In Section 1, we prepare some lemmas and theorems to prove the main result, and recall a cohomological Hasse principle and $\acute{\text{e}}$tale homology theory. In Section 2, we construct the map $\delta _Y$ and prove Theorem \ref{IT}. In the last of this paper, we give a simple normal crossing surface over a finite field whose reciprocity map $\rho _Y$ is potentially not injective. 
%notation%%%%%%%%%%%%%%%%%%%%%%%%%%%%%%%%%%%%%%%%%%%%%%%%%%%%%%%%%%%%%%%%%%%%%%%%%%%%%%%%%%%%%%%%%%%%%%%%%%%%%%%
%%%%%%%%%%%%%%%%%%%%%%%%%%%%%%%%%%%%%%%%%%%%%%%%%%%%%%%%%%%%%%%%%%%%%%%%%%%%%%%%%%%%%%%%%%%%%%%%%%%%%%%%%%%%%%%%
\subsection{notation}

(0.1) For an abelian group $A$ and a positive integer $n$, $A/n$ denotes the cokernel of the map $A\stackrel{\times n}{\longrightarrow }A$. $A_{tors}$ denotes the torsion subgroup of $A$. $A^{\oplus n}$ denotes the direct sum of $n$ copies of $A$. 

(0.2) For a field $k$, $k^{\times }$ denotes the multiplicative group, $k^{sep}$ denotes a fixed separable closure, $G_k$ denotes the absolute Galois group $\mathrm{Gal}(k^{sep}/k)$, $G_k^{ab}$ denotes the maximal abelian quotient group of $G_k$. For a connected scheme $X$, $\pi _1^{ab}(X)$ denotes the abelian $\acute{\text{e}}$tale fundamental group. Further, for a non-connected scheme $V$, $\pi_1^{ab}(V)$ denotes $\bigoplus_i \pi_1^{ab}(V_i)$ where $V_i$ are connected components of $V$. For $k$-scheme $X$, $\pi _1^{geo}(X)$ denotes $\mathrm{Ker}\bigl( \pi _1^{ab}(X)\longrightarrow G_k^{ab} \bigr)$. 

(0.3) Let $k$ be a field and $X$ be a $k$-scheme. For a field extension $F/k$, $X\otimes _kF$ denotes $X\times _{\mathrm{Spec}(k)}\mathrm{Spec}(F)$. Especially, for a fixed separable closure $k^{sep}/k$, $\overline{X}$ denotes $X\times _{\mathrm{Spec}(k)}\mathrm{Spec}(k^{sep})$. 

(0.4) For a scheme $X$ and an integer $q\geq 0$, $X^q$ denotes the set of points on $X$ of codimension $q$. If $X$ is of finite type over a field, $X_q$ denotes the set of points on $X$ which $\dim ( \overline{\{ x \}})=q$. Put $d:=\dim X$, $X_q=X^{d-q}$. For a point $x \in X$, $\kappa (x)$ denotes the residue field. For an integral scheme $X$, $k(X)$ denotes the function field. For a scheme $X$ of finite type over a field $k$ and of pure dimension $d$, we define the following group:
\begin{align*}
CH_0(X) : = \mathrm{Coker} \Bigl( \partial _1 : \bigoplus_{x \in X^{d-1}} \kappa (x)^{\times } \longrightarrow \bigoplus_{x \in X^d} \mb{Z} \Bigr),
\end{align*}
where $\partial _1$ is defined by the discrete valuation. 

If $X$ is proper over $k$, there is the degree map
\begin{align*}
\xymatrix{
CH_0(X) \ar[r]^{ \ \ \mathrm{deg}} &\mb{Z},
}
\end{align*}
$A_0(X)$ denotes its kernel. 

(0.5) $H^r(-,-)$ denotes an $\acute{\text{e}}$tale cohomology group. Especially, $H^r(F,-)$ denotes $H^r(\mathrm{Spec}(F), -)$ for a field $F$.

Let $X$ be a scheme. For an integer $n>1$ invertible on $X$, we denote $\mu _n$ the $\acute{\text{e}}$tale sheaf of $n$-th roots of unity. For an integer $i\geq  0$, we denote $\mb{Z}/n(i)$ the $\acute{\text{e}}$tale sheaf $\mu _n^{\otimes i}$.

Let $X$ be a smooth variety over a perfect field of positive characteristic $p$. For a positive integer $n=mp^r \ \bigl((m,p)=1\bigr)$ and a positive integer $i$, we put $\mb{Z}/n(i):=\mb{Z}/m(i)\bigoplus W_r\Omega _{X,log}^i[-i]$, where $W_r\Omega _{X,log}^{\bullet }$ denotes the logarithmic part of the de Rham$-$Witt complex $W_r\Omega _X^{\bullet  }$ on $X_{\acute{\text{e}}\text{t}}$ (cf. \cite{Il}).

(0.6) We recall here Bloch$-$Kato conjecture.
\begin{conj}[Bloch$-$Kato \cite{BK}]\label{BKC}
Let $k$ be a field and $i$ be non-negative integer. Then for a positive integer $n$ prime to the characteristic $ch(k)$ of $k$, the following Galois symbol map is bijective
\begin{equation*}
h_{k,n}^i : K_i(k)/n \longrightarrow H^i(k,\mathbb{Z}/n(i)).
\end{equation*}
\end{conj}
For this conjecture, the following results is known:
\begin{thm}\label{TBK}Conjecture $\ref{BKC}$ is true for the following cases:\\
\textup{(i)} $i\leq 2$, and $n$ is arbitrary positive integer prime to $ch(k)$.\\
\textup{(ii)} $n=2$, and $i$ is arbitrary non-negative integer. 
\end{thm}

The case $i=1$ follows from the Kummer theory, and the case $i=0$ is clear. The case $i=2$ is proved by Merkur'ev$-$Suslin \cite{MS}. The case (ii) is due to Voevodsky \cite{V}.
%%%Preliminary%%%%%%%%%%%%%%%%%%%%%%%%%%%%%%%%%%%%%%%%%%%%%%%%%%%%%%%%%%%%%%%%%%%%%%%%%%%%%%%%%%%%%%%%%%%%%%%%%%%%%%%%%%%%%%%%%%%%%%%%%%%%%%%%%%%%%
%%%%%%%%%%%%%%%%%%%%%%%%%%%%%%%%%%%%%%%%%%%%%%%%%%%%%%%%%%%%%%%%%%%%%%%%%%%%%%%%%%%%%%%%%%%%%%%%%%%%%%%%%%%%%%%%%%%%%%%%%%%%%%%%%%%%%%%%%%%%%%%%%%%
\section{Preliminary}

In this section, we prepare some lemmas and theorems for the proof of the main result. We also recall the cohomological Hasse principle and $\acute{\text{e}}$tale homology theory.

Through this paper, $k$ is a finite field and $n$ is a natural number. 

%SNC vra over finite fields%%%%%%%%%%%%%%%%%%%%%%%%%%%%%%%%%%%%%%%%%%%%%%%%%%%%%%%%%%%%%%%%%%%%%%%%%%%%%%%%%%%%%%%%%%%%%%%%%%%%%%%%%%%%%%%%
%%%%%%%%%%%%%%%%%%%%%%%%%%%%%%%%%%%%%%%%%%%%%%%%%%%%%%%%%%%%%%%%%%%%%%%%%%%%%%%%%%%%%%%%%%%%%%%%%%%%%%%%%%%%%%%%%%%%%%%%%%%%%%%%%%%%%%%%%%%
\subsection{Simple normal crossing varieties over finite fields}

We here prove some lemmas about simple normal crossing varieties over finite fields. In case of curves, similar lemmas are proved in \cite{MSA}. We extend these lemmas to higher dimensional cases. First we define a simple normal crossing variety over a field.

\begin{df}\label{d1}
Let $X$ be a equidimensional scheme of finite type over a field $k$. Then we call $X$ a \textit{normal crossing variety} over $k$, if $X$ is separated over $k$ and everywhere $\acute{\text{e}}$tale locally isomorphic to 
\begin{align*}
\mathrm{Spec}\bigl(k[T_0,\cdots, T_d]/(T_0T_1\cdots T_r)\bigr) \ \ \ (0 \leq r\leq d=\dim X).
\end{align*}

A normal crossing variety $X$ is called \textit{simple} if any irreducible component of $X$ is smooth over $k$. For a simple normal crossing variety $X$, we use the following notation: Let $\{ X_i\}_{i \in I}$ be the set of irreducible components of $X$. For a positive integer $r$, we define
\begin{align*}
X^{(r)} := \coprod _{\{ i_1, i_2, \dots ,i_r \} \subset I} X_{i_1}\times _XX_{i_2}\times _X \dots \times _XX_{i_r}.
\end{align*}
\end{df}

Now we define a simplicial complex of which homology groups are very important tool in the unramified class field theory for simple normal crossing varieties over finite fields.

\begin{df}\label{d2}
Let $X$ be a $d$-dimensional simple normal crossing variety over $k$. Then we define a simplicial complex $\Gamma _X$ called the \textit{dual graph} of $X$ as follows:

Let $\{ X_i \}_{i \in I}$ be the set of irreducible components of $X$. Fix an ordering on $I$. The set of $r$-simplexes $\mf{S}_r$ of $\Gamma _X$ is the set of irreducible components of $X^{(r)}$. We determine the orientation on $r$-simplexes inductively on $r$ by the fixed ordering on $I$ (cf. \cite[$\S3$]{JS}). 

Let $F/k$ be an algebraic extension. We put $Y:=X\otimes _kF$. Let $\{ Y_j\}_{j\in J}$ be the set of irreducible components of $Y$. Then we define a semi-order on $J$ as follows: $j_1, j_2 \in J,$
\begin{align*}
j_1 < j_2 \Longleftrightarrow \phi (j_1) < \phi (j_2),
\end{align*}
where $\phi : J \longrightarrow I$ is the map which sends $j$ to $\phi (j)$ when $Y_j$ lies above $X_{\phi (j)}$. By using this order on $J$, we define the homomorphism of the complexes
\begin{align*}
\sigma_{F/k} :\Gamma _Y \longrightarrow \Gamma _X.
\end{align*}
Then the homomorphism $H_a(\Gamma _Y, \mb{Z})\longrightarrow H_a(\Gamma _X,\mb{Z})$ induced by $\sigma _{F/k}$ is called \textit{norm map}. 
\end{df}

In the rest of this subsection, $X$ denotes a simple normal crossing variety over a finite field $k$ which is proper over $k$.

\begin{lem}\label{L1}
\textup{(1)} The degree `$0$-part' $A_0(X)$ of $CH_0(X)$ is finite.\\
\textup{(2)} Assume that each connected component of $X^{(2)}$ has a $k$-rational point. Then the canonical map
 $\iota : \bigoplus_{i \in I} A_0(X_i) \longrightarrow A_0(X)$ is surjective. 
\end{lem}
\begin{proof}We consider the following commutative diagram with exact rows:
\begin{align}\label{a1}
\xymatrix{
& & CH_0(X^{(2)}) \ar[r]^{deg_{X^{(2)}/k}} \ar[d]&\bigoplus K_0(k) \ar[d] \ar[r] &\bigoplus \mb{Z}/m_l \ar[d]^{\nu }\rightarrow 0\\
0 \ar[r] & \bigoplus_i A_0(X_i) \ar[r] \ar[d]^{\iota } & \bigoplus_i CH_0(X_i)\ar[r]^{deg_{X^{(1)}/k}} \ar[d] & \bigoplus_i K_0(k) \ar[d] \ar[r] & \bigoplus \mb{Z}/m_i\rightarrow 0\\
0 \ar[r] & A_0(X) \ar[r] & CH_0(X) \ar[r] & K_0(k).
}
\end{align}
Here $m_l = [ \Gamma (X_l^{(2)},\ml{O}_{X_l^{(2)}}):k]$ for connected components $X_l^{(2)}$ of $X^{(2)}$ and $m_i = [ \Gamma (X_i,\ml{O}_{X_i}):k]$.

(1) Since $X_i$ is smooth for any $i$, $A_0(X_i)$ is finite by a theorem of Kato-Saito \cite{KS}.

On the other hand, by the diagram we have a surjective map from the kernel of $\nu $ to the cokernel of $\iota$. Hence we see that the cokernel of the map $\iota $ is finite by and that $A_0(X)$ is finite.

(2) Under the assumption, the map $deg_{X^{(2)}/k}$ is surjective. The assertion follows from the diagram \eqref{a1}. 
\end{proof}
\begin{lem}\label{L2}
Let $\{ W_s \}$ be the set of connected components of $\overline{X}$, and $\{ V_j \}$ be the set of irreducible components of $\overline{X}$. Then there is an exact sequence of finite left $G_k$-modules :
\begin{align}\label{a2}
\xymatrix{
\bigoplus \pi _1^{ab}(V_j)/n \ar[r] &\bigoplus \pi _1^{ab}(W_s)/n \ar[r] &H_1(\Gamma _{\overline{X}},\mb{Z}/n) \ar[r] &0.
}
\end{align}
Further the following diagram commutes:
\begin{align}\label{a3}
\xymatrix{
\bigoplus \pi_1^{ab}(W_s)/n \ar[r] \ar[d]_{(*1)} & \pi_1^{ab}(X)/n \ar[d]_{(*2)}\\
H_1(\Gamma _{\overline{X}}, \mb{Z}/n) \ar[r] & H_1(\Gamma _X, \mb{Z}/n).
}
\end{align}
\end{lem}
\begin{proof}
Since we have the canonical isomorphism 
\begin{align*}
\bigoplus \pi _1^{ab}(W_s)/n\simeq \mathrm{Hom}\bigl( H^1(\overline{X}, \mb{Z}/n), \mb{Q}/\mb{Z} \bigr),
\end{align*}
it is sufficient to prove the corresponding statements on $\acute{\text{e}}$tale cohomology groups and `cohomology' of dual graphs. The finiteness of groups in \eqref{a2} follows from the finiteness of $\acute{\text{e}}$tale cohomology groups \cite[XVI, 5.2]{SGA4} and the definition of dual graph.

We consider the following exact sequence of $\acute{\text{e}}$tale sheaves on $\overline{X}_{\acute{\text{e}}t}$:
\begin{align*}
\xymatrix{
0 \ar[r] & \mb{Z}/n_{\overline{X}} \ar[r] & \bigoplus \mb{Z}/n_{V_j} \ar[r] & \cdots \ar[r]& \bigoplus\mb{Z}/n_{{\overline{X}}^{(d+1)}_t} \ar[r] &0.
}
\end{align*}
Here ${\overline{X}}^{(d+1)}_t$ denotes irreducible components of ${\overline{X}}^{(d+1)}$, and we have omitted the indication of direct image functors of sheaves. From this exact sequence, we obtain a spectral sequence
\begin{align*}
E_1^{p,q}=H^q({\overline{X}}^{(p)}, \mb{Z}/n) \Longrightarrow H^{p+q}(\overline{X}, \mb{Z}/n).
\end{align*}
By computing $E_2$-terms, we have an exact sequence
\begin{align*}
\xymatrix{
0 \ar[r] & H^1(\Gamma _{\overline{X}},\mb{Z}/n) \ar[r] & H^1(\overline{X}, \mb{Z}/n) \ar[r] & \bigoplus_{j}H^1(V_j, \mb{Z}/n). 
}
\end{align*}
The Pontryagin dual of this sequence provides us with the map $(*1)$ and proves the exactness of \eqref{a2}. The commutativity of \eqref{a3} follows from the following commutative diagram of $\acute{\text{e}}$tale sheaves on $X_{\acute{\text{e}}\text{t}}$
\begin{align*}
\xymatrix{
0 \ar[r] & \mb{Z}/n_{X} \ar[r] \ar[d] & \bigoplus \mb{Z}/n_{X_i} \ar[r] \ar[d] & \cdots \ar[r] & \mb{Z}/n_{X^{(d+1)}} \ar[r] \ar[d]&0\\
0 \ar[r] & \mb{Z}/n_{\overline{X}} \ar[r] & \bigoplus \mb{Z}/n_{V_j} \ar[r] & \cdots \ar[r]& \mb{Z}/n_{{\overline{X}}^{(d+1)}} \ar[r] &0,
}
\end{align*}
and the fact that the map $(*2)$ comes from the upper row. 
\end{proof}

The following lemma follows from the Hochshild-Serre spectral sequence associated with $\overline{X}\rightarrow X$ (cf. \cite{Mi}).

\begin{lem}\label{L3}
The canonical map $\bigoplus \pi_1^{ab}(W_s)\longrightarrow \pi_1^{ab}(X)$ induces an isomorphism
\begin{align*}
\bigoplus \pi_1^{ab}(W_s)_{G_k}\simeq  \pi_1^{geo}(X).
\end{align*}
\end{lem}

We put $\displaystyle H_1(\Gamma _X, \hat{\mb{Z}}):= \lim_{\longleftarrow n} H_1(\Gamma _X, \mb{Z}/n)$. We write $H_1(\Gamma _X, \hat{\mb{Z}})_{\overline{X}}$ for the image of the norm map $H_1(\Gamma _{\overline{X}}, \hat{\mb{Z}}) \longrightarrow H_1(\Gamma _X, \hat{\mb{Z}})$. The following proposition is the `geometric' part of the unramified class field theory for a simple normal crossing variety. 

\begin{prop}\label{P1}
There is an exact sequence 
\begin{align*}
\xymatrix{
A_0(X) \ar[r]^{\rho _X^{geo}} & \pi_1^{geo}(X) \ar[r]^{(*3)} & H_1(\Gamma _X, \hat{\mb{Z}})_{\overline{X}} \ar[r] &0.
}
\end{align*}
Further the following diagram commutes:
\begin{align*}
\xymatrix{
\pi_1^{geo}(X) \ar[r] \ar[d]_{(*3)} & \pi_1^{ab}(X) \ar[d]_{(*2)}\\
H_1(\Gamma _{X}, \hat{\mb{Z}})_{\overline{X}} \ar[r] & H_1(\Gamma _X, \hat{\mb{Z}}).
}
\end{align*}
\end{prop}
\begin{proof}
We write $CH_0(X)^{\wedge }$ for $\displaystyle\lim_{\longleftarrow n} CH_0(X)/n$, $K_0(k)^{\wedge }$ for $\displaystyle\lim_{\longleftarrow n} K_0(k)/n$. We consider the following commutative diagram with exact rows:
\begin{align*}
\xymatrix{
0\ar[r] & A_0(X) \ar[r] \ar[d]_{\rho _X^{geo}} &CH_0(X)^{\wedge } \ar[r]^{deg_{X/k}^{\wedge }} \ar[d]_{\rho _X^{\wedge }} & K_0(k)^{\wedge }\ar[d]_{\rho _k^{\wedge }}\\
0\ar[r] & \pi_1^{geo}(X) \ar[r] &\pi_1^{ab}(X) \ar[r] \ar[d]_{(*2)} &G_k \ar[r] &0\\
& & H_1(\Gamma _X, \hat{\mb{Z}}),
}
\end{align*}
where the map $\rho _X^{\wedge }$ is induced by the reciprocity map $\rho _X$ of the unramified class field theory for $X$. Here we used the finiteness of $A_0(X)$ (cf. Lemma \ref{L1}), and the fact that the pro-finite completion of an exact sequence of finitely generated abelian groups is exact. Since the map $\rho _k^{\wedge }$ is injective (in fact bijective), the cokernel of the map $\rho _X^{geo}$ is isomorphic to the image of $\pi_1^{geo}(X)$ in $H_1(\Gamma _X, \hat{\mb{Z}})$. Therefore $\mathrm{Coker}\bigl( \rho _X^{geo}\bigr)$ coincides with $H_1(\Gamma _X, \hat{\mb{Z}})_{\overline{X}}$ from Lemma \ref{L2}. 
\end{proof}
\begin{rmk}\label{R1}
If $X$ is a curve, the map $\rho_X^{geo}$ in the above proposition is injective by Kato$-$Saito \cite{KS}.
\end{rmk}
%Bloch-Ogus-Kato complex%%%%%%%%%%%%%%%%%%%%%%%%%%%%%%%%%%%%%%%%%%%%%%%%%%%%%%%%%%%%%%%%%%%%%%%%%%%%%%%%%%%%%%%%%%%%%%%%%%%%%%%%%%%%%%%%%%%
%%%%%%%%%%%%%%%%%%%%%%%%%%%%%%%%%%%%%%%%%%%%%%%%%%%%%%%%%%%%%%%%%%%%%%%%%%%%%%%%%%%%%%%%%%%%%%%%%%%%%%%%%%%%%%%%%%%%%%%%%%%%%%%%%%%%%%%%%%%
\subsection{Bloch$-$Ogus$-$Kato complex}

In this subsection, we recall the theorem called a cohomological Hasse principle. First we recall Bloch$-$Ogus$-$Kato complex. 

For $X$ be an excellent scheme and integers $r,s,n>0$, Kato define a homological complex $C^{r,s}(X,n)$ (cf. \cite[\S 1]{K}):
\begin{align*}
\cdots \rightarrow \bigoplus_{x \in X_i}&H^{r+i}(\kappa (x),\mb{Z}/n(s+i))\rightarrow \bigoplus_{x \in X_{i-1}}H^{r+i-1}(\kappa (x),\mb{Z}/n(s+i-1))\\
&\rightarrow \cdots \rightarrow \bigoplus_{x \in X_1}H^{r+1}(\kappa (x),\mb{Z}/n(s+1))\rightarrow \bigoplus_{x \in X_0}H^{r}(\kappa (x),\mb{Z}/n(s)).
\end{align*}
The degree of the term $\displaystyle \bigoplus_{x \in X_i}$ is $i$. The following canonical map for a fraction field $K$ of a discrete valuation ring and its residue field $F$ plays an important role in the definition of the boundary map for the above complex
\begin{align*}
H^i(K,\mb{Z}/n(j))\longrightarrow H^{i-1}(F,\mb{Z}/n(j-1)).
\end{align*}
\begin{rmk}In the definition of $C^{r,s}(X,n)$, we assume that if $r=s+1$, for any prime divisor $p$ of $n$ and any $x \in X_0$ such that $\mathrm{ch}(\kappa (x))=p$, we have $[\kappa (x):\kappa (x)^p]\leq p^s$. The assumption is satisfied in the case we consider in this paper.
\end{rmk}
\begin{df}\label{d3}
Kato homology of $X$ with coefficient $\mb{Z}/n$ defined by
\begin{align*}
H_i^{r,s}(X,n):=H_i(C^{r,s}(X,n)).
\end{align*}
\end{df}

Kato conjectured the following for varieties over $k$ :
\begin{conj}[Kato \cite{K}]\label{KC}
Let $X$ be a connected projective smooth variety over $k$. Put $H_i^K(X,n):=H_i^{1,0}(X,n)$. Then we have
\begin{align*}
H_i^K(X,n)\simeq \begin{cases}
                   0  \ \ & \text{if $i \not = 0$},\\
                   \mb{Z}/n  \ \ & \text{if $i = 0$}.
                 \end{cases}
\end{align*}
\end{conj}

If $\dim X=1$, this is classical. This is an analogy of the following exact sequence for Brauer group of a number field $K$:
\begin{align*}
0\longrightarrow \mathrm{Br}(K) \longrightarrow \bigoplus_{v}\mathrm{Br}(K_v)\longrightarrow \mb{Q}/\mb{Z}\longrightarrow 0 .
\end{align*}

The following theorem is called the cohomological Hasse principle and shows Conjecture \ref{KC} is true for $\dim X=2$. Colliot-Th$\acute{\text{e}}$l$\grave{\text{e}}$ne$-$Sansuc$-$Soul$\acute{\text{e}}$ \cite{CTSS} (in the prime to $ch(k)$ case) and Kato \cite{K} (in the $ch(k)$-primary case) prove the theorem. 
\begin{thm}[Colliot-Th$\acute{\text{e}}$l$\grave{\text{e}}$ne$-$Sansuc$-$Soul$\acute{\text{e}}$, Kato]\label{TK1}
Let $X$ be a proper smooth irreducible surface over $k$ and $n$ be a natural number. Let $\eta $ be the generic point of $X$. Then Bloch$-$Ogus$-$Kato complex
\begin{align*}
0\rightarrow H^3(k(X) ,\mb{Z}/n(2))\rightarrow  \bigoplus_{x \in X_1} H^2(\kappa (x),\mb{Z}/n(1))\rightarrow  \bigoplus_{x \in X_0}H^1(\kappa (x),\mb{Z}/n) 
\end{align*}
is exact and the cokernel of the last map isomorphic to $\mb{Z}/n$.
\end{thm}

We know the following theorem on the Kato Conjecture replaced coefficient $\mb{Z}/n$ by $\mb{Q}_{\ell }/\mb{Z}_{\ell }$ for a prime number $\ell $, which is proved by Colliot-Th$\acute{\text{e}}$l$\grave{\text{e}}$ne \cite{CT} in the case $\ell $ prime to $ch(k)$, and by Suwa \cite{Sw} in the case $\ell $ is a power of $ch(k)$.
\begin{thm}[Colliot-Th$\acute{\text{e}}$l$\grave{\text{e}}$ne, Suwa]\label{TK2}
For any prime number $\ell $, $\mb{Q}_{\ell }/\mb{Z}_{\ell }$-coefficient Kato conjecture holds true for degree $i\leq 3$.
\end{thm}

Bloch$-$Kato conjecture (Conjecture \ref{BKC}) relates to $\mb{Z}/{\ell }^{\nu }$-coefficient Kato conjecture and $\mb{Q}_{\ell }/\mb{Z}_{\ell }$-coefficient Kato conjecture as follows: If $\mb{Q}_{\ell }/\mb{Z}_{\ell }$- coefficient Kato conjecture holds true for degree $i\leq m$ and Bloch$-$Kato conjecture holds true for degree $i\leq m$, then  $\mb{Z}/{\ell }^{\nu }$-coefficient Kato conjecture holds true for degree $i\leq m$.\bigskip

Jannsen and Saito \cite{JS} proved the following theorem which shows a relation between homology groups of a dual graph and Kato homology groups.
\begin{thm}[Jannsen$-$Saito]\label{TJS}
Let $X$ be a $d$ dimensional simple normal crossing variety over $k$. Let $\ell $ be a prime number and $\nu $ be a natural number. Then there is a canonical homomorphism 
\begin{align*}
\gamma _a : H_a^K(X,\ell ^{\nu })\longrightarrow H_a(\Gamma _X,\mb{Z}/\ell ^{\nu }).
\end{align*}
Further, if Conjecture $\ref{KC}$ holds true for degree $\leq m$, then $\gamma _a$ is isomorphism for any $a\leq m$. 
\end{thm}
%etale homology theory%%%%%%%%%%%%%%%%%%%%%%%%%%%%%%%%%%%%%%%%%%%%%%%%%%%%%%%%%%%%%%%%%%%%%%%%%%%%%%%%%%%%%%%%%%%%%%%%%%%%%%%%%%%%%%%%%%%%%%%%%%%%%%%%%
%%%%%%%%%%%%%%%%%%%%%%%%%%%%%%%%%%%%%%%%%%%%%%%%%%%%%%%%%%%%%%%%%%%%%%%%%%%%%%%%%%%%%%%%%%%%%%%%%%%%%%%%%%%%%%%%%%%%%%%%%%%%%%%%%%%%%%%%%%%%%%%%%%%%%%%
\subsection{$\acute{\text{E}}$tale homology theory}

We here define $\acute{\text{e}}$tale homology and compute a spectral sequence associated with that homology.

For a separated scheme $X$ of finite type over $k$ and a natural number $n$, we define the $\acute{\text{e}}$tale homology with coefficient $\mb{Z}/n$ to
\begin{align*}
H_i(X,\mb{Z}/n):= \mathrm{Hom}\bigl( H_c^i(X, \mb{Z}/n),\mb{Q}/\mb{Z} \bigr).
\end{align*}
Here $H_c^i(-,-)$ denotes an $\acute{\text{e}}$tale cohomology group with compact support. This functor $H_*(-,\mb{Z}/n)$ forms a homology theory on the category of separated schemes of finite type over $k$ and proper $k$-morphisms. By Bloch$-$Ogus \cite{BO}, we have a niveau spectral sequence 
\begin{align}\label{a4}
E_{p,q}^1=\bigoplus_{x \in X_p}H_{p+q}(x ,\mb{Z}/n)\Longrightarrow H_{p+q}(X,\mb{Z}/n).
\end{align}
Here for $x \in X$, we define
\begin{align*}
H_i(x ,\mb{Z}/n):=\lim_{\longrightarrow }H_i(U,\mb{Z}/n),
\end{align*}
and the limit is taken over all nonempty $U$ which is open in the closure $\overline{ \{ x \}}$ of $x$ in $X$.\bigskip

This homology theory satisfies the Poincar$\acute{\text{e}}$ duality:
\begin{lem}\label{L4}
Let $V$ be an connected smooth variety over $k$ with $\dim=d$. Let $n$ be a natural number. Then we have the following isomorphism:
\begin{align*}
H_m(V,\mb{Z}/n)\simeq H^{2d+1-m}(V, \mb{Z}/n(d)).
\end{align*}
\end{lem}
\begin{proof}
If $n$ prime to $ch(k)$, this isomorphism follows from the Poincar$\acute{\text{e}}$ duality \cite{Mi} for $\overline{V}$ and the duality for Galois cohomology of $k$. 

If $n$ is a power of $ch(k)$, this is proved in \cite{JSS}. 
\end{proof}

The following proposition follows from Lemma \ref{L4} and the above spectral sequence \eqref{a4} (cf. \cite[Prop. 2.4]{MSA}).
\begin{prop}\label{P2}
For a variety $X$ over $k$ and any $n \in \mb{N}$, we have a spectral sequence
\begin{align}\label{P2-ss}
E_{p,q}^1(X,n):=\bigoplus_{x \in X_p}H^{p-q+1}(\kappa (x),\mb{Z}/n(p))\Longrightarrow H_{p+q}(X,\mb{Z}/n).
\end{align}
If $p<0$ or $q<0$, then $E_{p,q}^1=0$. Further, $E^1$-terms are Bloch$-$Ogus$-$Kato complexes. 
\end{prop}

$E_{p,q}^1(X)$ denotes $E_{p,q}^1(X,n)$ below.

\begin{prop}\label{P2-2}
Let $X$ be a simple normal crossing variety over $k$ which is proper over $k$. For any integer $n>1$, there is an exact sequence
\begin{align*}
\xymatrix{
H_2(\Gamma _X, \mb{Z}/n)\ar[r]^{\epsilon _{X,n}} &CH_0(X)/n \ar[r]^{\rho_ X/n}&\pi_1^{ab}(X)/n \ar[r] &H_1(\Gamma _X, \mb{Z}/n) \ar[r] &0.
}
\end{align*}
\end{prop}

\begin{proof}
From the spectral sequence \eqref{P2-ss} for X, we obtain the following exact sequence
\begin{align*}
\xymatrix{
E^2_{2,0}(X)\ar[r] &E^2_{0,1}(X) \ar[r]^{d_{0,1}}&H_1(X,\mb{Z}/n) \ar[r] &E^2_{1,0}(X) \ar[r] &0.
}
\end{align*}
By Theorem \ref{TBK}, \ref{TK1} and \ref{TK2}, Conjecture \ref{KC} holds true for degree $i\leq 2$. Therefore, from Theorem \ref{TJS}, we obtain  isomorphisms
\begin{align*}
&E^2_{2,0}(X) = H^K_2(X,n) \simeq H_2(\Gamma _X, \mb{Z}/n),\\
&E^2_{1,0}(X) = H^K_1(X,n) \simeq H_1(\Gamma _X, \mb{Z}/n).
\end{align*}
On the other hand, there is a commutative diagram
\begin{align*}
\xymatrix{
E^2_{0,1}(X) \ar[d]^{\simeq }\ar[r]^{d_{0,1}}&H_1(X,\mb{Z}/n)\ar[d]^{\simeq }\\
CH_0(X)/n\ar[r]^{\rho_X/n} &\pi_1^{ab}(X)/n.
}
\end{align*}
Here the vertical isomorphisms follow from the definition of $CH_0(X)$ and $H^1(X,\mb{Z}/n)=\mathrm{Hom}\bigl( \pi_1^{ab}(X), \mb{Z}/n \bigr)$. Hence the proposition follows. 
\end{proof}
%%%Main result%%%%%%%%%%%%%%%%%%%%%%%%%%%%%%%%%%%%%%%%%%%%%%%%%%%%%%%%%%%%%%%%%%%%%%%%%%%%%%%%%%%%%%%%%%%%%%%%%%%%%%%%%%%%%%%%%%%%%%%%%%%%%%%%%%%%%%%
%%%%%%%%%%%%%%%%%%%%%%%%%%%%%%%%%%%%%%%%%%%%%%%%%%%%%%%%%%%%%%%%%%%%%%%%%%%%%%%%%%%%%%%%%%%%%%%%%%%%%%%%%%%%%%%%%%%%%%%%%%%%%%%%%%%%%%%%%%%%%%%%%%%%%
\section{Main result}

In this section, we construct the map $\delta _Y$ of introduction and prove Theorem \ref{IT}. Let $Y_0$ be a projective smooth and geometrically irreducible variety over $k$ and let $D$ be a simple normal crossing divisor on $Y_0$. We then consider the following simple normal crossing variety 
\begin{align*}
Y := \bigl( Y_0\times _kO \bigr) \cup \bigl( Y_0\times _k{\infty }\bigr) \cup \bigl( D\times _k\mb{P}^1_k\bigr) \ \ \subset Y_0\times _k\mb{P}^1_k.
\end{align*}
Here $O:=(0:1)$, $\infty :=(1:0)\in \mb{P}^1_k$.
%%%Construction of \delt%%%%%%%%%%%%%%%%%%%%%%%%%%%%%%%%%%%%%%%%%%%%%%%%%%%%%%%%%%%%%%%%%%%%%%%%%%%%%%%%%%%%%%%%%%%%%%%%%%%%%%%%%%%%%%%%%%%%%%%%%%%%
%%%%%%%%%%%%%%%%%%%%%%%%%%%%%%%%%%%%%%%%%%%%%%%%%%%%%%%%%%%%%%%%%%%%%%%%%%%%%%%%%%%%%%%%%%%%%%%%%%%%%%%%%%%%%%%%%%%%%%%%%%%%%%%%%%%%%%%%%%%%%%%%%%%%
\subsection{Construction of $\delta _Y$}
First, we prove the following proposition on an important homomorphism in study on $\mathrm{Ker}(\rho_Y)$.
\begin{prop}\label{P3}
There exists a homomorphism 
\begin{align*}
\delta _Y : H_1(\Gamma _D,\mb{Z})\longrightarrow CH_0(Y)
\end{align*}
whose image coincides with $\mathrm{Ker}(\rho _Y)$.
\end{prop}

We prove this proposition admitting the following lemma.
\begin{lem}\label{L5}
There exists an exact sequence:
\begin{align*}
\xymatrix{
H_1(\Gamma _D,\mb{Z}/n) \ar[r]^{\delta _n} &CH_0(Y)/n \ar[r]^{\rho_Y/n} &\pi_1^{ab}(Y)/n.
}
\end{align*}
\end{lem}

\renewcommand{\proofname}{\textit{Proof of Proposition $\ref{P3}$}}
\begin{proof}We consider the following diagram:
\begin{align}\label{p61}
\xymatrix{
H_1(\Gamma _D,\mb{Z}) \ar[d] \ar@{.>}[r] & CH_0(Y) \ar[r]^{\rho _Y} \ar[d] & \pi_1^{ab}(Y)\ar@2{-}[d]\\
H_1(\Gamma _D, \hat{\mb{Z}}) \ar[r]^{\hat{\delta }} & CH_0(Y)^{\wedge } \ar[r]^{\rho _Y^{\wedge }} &\pi _1^{ab}(Y),
}
\end{align}
where the bottom sequence is obtain by the projective limit of Lemma \ref{L5} and exact. Since $A_0(Y)$ is finite (cf. Lemma \ref{L1}(1)) and $\mathrm{Ker}(\rho _Y)\subset A_0(Y)$, we have $\mathrm{Ker}(\rho _Y)\simeq \mathrm{Ker}(\rho _Y^{\wedge })$. We define $\delta _Y$ by the composite
\begin{align*}
\xymatrix{
H_1(\Gamma _D, \mb{Z})\ar[r] &H_1(\Gamma _D, \hat{\mb{Z}}) \ar[r]^{\hat{\delta }} & \mathrm{Ker}(\rho _Y^{\wedge })\ar[r]^{\sim } &\mathrm{Ker}(\rho _Y) \ar@{^{(}->}[r]& CH_0(Y).
}
\end{align*}
Then we have $\mathrm{Im}(\delta _Y)=\mathrm{Ker}(\rho _Y)$, since the group $\mathrm{Im}(\hat{\delta })=\mathrm{Ker}(\rho _Y^{\wedge })$ is finite and the map $H_1(\Gamma _D, \mb{Z})\rightarrow H_1(\Gamma _D, \hat{\mb{Z}})$ has dense image with respect to the pro-finite topology. 
\end{proof}
\renewcommand{\proofname}{\textit{Proof}}

We consider the norm map $\sigma : H_1(\Gamma _{\overline{D}},\mb{Z})\rightarrow H_1(\Gamma _D,\mb{Z})$ and put 
\begin{align*}
G(Y):= \mathrm{Im}\bigl(\delta _Y\circ \sigma : H_1(\Gamma _{\overline{D}}, \mb{Z})\longrightarrow CH_0(Y) \bigr).
\end{align*}
The group $G(Y)$ is a subgroup of $\mathrm{Ker}\bigl( \rho _Y\bigr)$ from Proposition \ref{P3}. In the next subsection, we describe $G(Y)$ with the map $\alpha $ defined below.\bigskip

We prove Lemma \ref{L5} admitting the following sublemma.
\begin{slem}\label{L6}
There are two exact sequences:\\
\textup{(1)} $CH_0(D)/n\longrightarrow CH_0(Y_0)/n^{\oplus 2} \longrightarrow CH_0(Y)/n\longrightarrow 0$, \\
\textup{(2)} $\pi _1^{ab}(D)/n\longrightarrow \pi _1^{ab}(Y_0)/n^{\oplus 2} \longrightarrow \pi _1^{ab}(Y)/n$. 
\end{slem}
\renewcommand{\proofname}{\textit{Proof of Lemma $\ref{L5}$}}
\begin{proof}
From Proposition \ref{P2-2} and Sublemma \ref{L6}, we obtain the following commutative diagram with exact rows:
\begin{align*}
\xymatrix{
CH_0(D)/n \ar[r] \ar[d] & CH_0(Y_0)/n^{\oplus 2}\ar[r] \ar[d]_{\simeq  } & CH_0(Y)/n \ar[r] \ar[d]_{\rho _Y/n} &0\\
\pi_1^{ab}(D)/n \ar[r] \ar@{->>}[d] &\pi_1^{ab}(Y_0)/n^{\oplus 2}\ar[r]& \pi_1^{ab}(Y)/n\\
H_1(\Gamma _D, \mb{Z}/n).
}
\end{align*}
By this diagram, we have a homomorphism
\begin{align*}
\delta _n : H_1(\Gamma _D, \mb{Z}/n)\longrightarrow CH_0(Y)/n
\end{align*}
whose image coincides with $\mathrm{Ker}(\rho _Y/n)$. 
\end{proof}

\begin{rmk}\label{R2}By Proposition \ref{P2-2}, we have an exact sequence
\begin{align*}
\xymatrix{
H_2(\Gamma _Y, \mb{Z}/n)\ar[r]^{\epsilon _{Y,n}} &CH_0(Y)/n \ar[r]^{\rho_Y/n}&\pi_1^{ab}(Y)/n.
}
\end{align*}
On the other hand, from the structure of $Y$, we obtain a suspension isomorphism
\begin{align*}
H_2(\Gamma _Y, \mb{Z}/n)\simeq H_1(\Gamma _D, \mb{Z}/n).
\end{align*}
Therefore we have the map
\begin{align*}
H_1(\Gamma _D, \mb{Z}/n)\simeq H_2(\Gamma _Y, \mb{Z}/n)\longrightarrow CH_0(Y)/n
\end{align*}
whose image coincides with $\mathrm{Ker}(\rho_Y/n)$. This map coincides with the map $\delta _n$ of Lemma \ref{L5}.
\end{rmk}\bigskip

To prove Sublemma \ref{L6}, we consider the following variety and two closed subschemes: 
\begin{align*}
&S:=\bigl( Y_0\times _kO \bigr)\sqcup\bigl( Y_0\times _k{\infty } \bigr)\sqcup\bigl( D\times _k\mb{P}^1 \bigr) ,\\
&Z:=\bigl( Y_0\times _kO \bigr)\cup \bigl( Y_0\times _k{\infty } \bigr) \subset Y ,\\
&Z^{\prime }:=\bigl( Y_0\times _kO \bigr)\sqcup \bigl( Y_0\times _k{\infty } \bigr)\sqcup \bigl( D\times _k \{ O ,\infty \} \bigr) \subset S.
\end{align*}
Then we have
\begin{align}\label{U}
Y\backslash Z \cong S\backslash Z^{\prime }\simeq D\times \mb{G}_m.
\end{align}
\renewcommand{\proofname}{\textit{Proof of Sublemma $\ref{L6}$}}
\begin{proof}
From Proposition \ref{P2} and \eqref{U}, we obtain two exact sequences of complexes for fixed $q$ :
\begin{align}\label{p51}
\xymatrix{
0 \ar[r] &E_{\ast ,q}^1(Z) \ar[r] &E_{\ast ,q}^1(Y) \ar[r] &E_{\ast ,q}^1(U) \ar[r] &0,
}
\end{align}
\begin{align}\label{p52}
\xymatrix{
0 \ar[r] &E_{\ast ,q}^1(Z^{\prime }) \ar[r] &E_{\ast ,q}^1(S) \ar[r] &E_{\ast ,q}^1(U) \ar[r] &0.
}
\end{align}

(1) From the above exact sequences \eqref{p51} and \eqref{p52} for $q=1$, we obtain the following commutative diagram with exact rows:
\begin{align}\label{a5}
\xymatrix{
E_{11}^2(D\times \mb{G}_m) \ar[r] \ar@2{-}[d] & CH_0(Z^{\prime })/n \ar[r]^{\beta } \ar[d]& CH_0(S)/n \ar[r] \ar[d] &CH_0(D\times \mb{G}_m) \ar@2{-}[d]\\
E_{11}^2(D\times \mb{G}_m) \ar[r] & CH_0(Z)/n \ar[r]& CH_0(Y)/n \ar[r]&CH_0(D\times \mb{G}_m)
}
\end{align}
Now we have $CH_0(D\times \mb{G}_m)=0$. We compute the kernel of the map $\beta $. Since there are the following isomorphisms
\begin{align*}
&CH_0(Z^{\prime })\simeq CH_0(Y_0)^{\oplus 2}\oplus CH_0(D)^{\oplus 2},\\
&CH_0(S)\simeq CH_0(Y_0)^{\oplus 2}\oplus CH_0(D\times \mb{P}^1),\\
&CH_0(D\times \mb{P}^1)\simeq CH_0(D),
\end{align*}
we have
\begin{align*}
\mathrm{Ker}(\beta )&=\bigl\{(0,0,c,-c) \bigl| \ c \in CH_0(D)/n \bigr\}\\
&\simeq CH_0(D)/n.
\end{align*}
Hence, by the diagram \eqref{a5} and $CH_0(Z)/n\simeq CH_0(Y_0)/n^{\oplus 2}$, we have an exact sequence
\begin{align*}
CH_0(D)/n\longrightarrow CH_0(Y_0)/n^{\oplus 2}\longrightarrow CH_0(Y)/n\longrightarrow 0.
\end{align*}

(2) Considering the localization sequence of $\acute{\text{e}}$tale homology groups, we obtain the following commutative diagram
{\small \begin{align}\label{a6}
\xymatrix{
H_2(D\times \mb{G}_m,\mb{Z}/n) \ar[r] \ar@2{-}[d] & \pi_1^{ab}(Z^{\prime })/n \ar[r]^{\beta^{\prime } } \ar[d]& \pi_1^{ab}(S)/n \ar[r] \ar[d] &H_1(D\times \mb{G}_m, \mb{Z}/n) \ar@2{-}[d]\\
H_2(D\times \mb{G}_m, \mb{Z}/n) \ar[r] & \pi_1^{ab}(Z)/n \ar[r]& \pi_1^{ab}(Y)/n \ar[r]& H_1(D\times \mb{G}_m, \mb{Z}/n).
}
\end{align}}
Similarly to (1), we have $\mathrm{Ker}\bigl(\beta ^{\prime }\bigr)\simeq \pi_1^{ab}(D)/n$. Hence, by the diagram \eqref{a6} and $\pi_1^{ab}(Z)/n\simeq \pi_1^{ab}(Y_0)/n^{\oplus 2}$, we have an exact sequence
\begin{align*}
\pi_1^{ab}(D)/n\longrightarrow \pi_1^{ab}(Y_0)/n^{\oplus 2}\longrightarrow \pi_1^{ab}(Y)/n. 
\end{align*}
\end{proof}

We use the following lemma to define the map $\delta _Y^{geo}$ below.
\begin{lem}\label{L7}
There are two exact sequences:\\
\textup{(1)} $A_0(D)\longrightarrow A_0(Y_0)^{\oplus 2} \longrightarrow A_0(Y)$,\\
\textup{(2)} $\pi _1^{geo}(D)\longrightarrow \pi _1^{geo}(Y_0)^{\oplus 2} \longrightarrow \pi _1^{geo}(Y)$. 
\end{lem}
\renewcommand{\proofname}{\textit{Proof}}
\begin{proof}
From the projective limit of Sublemma \ref{L6}, we have the following commutative diagram with exact rows:
\begin{align*}
\xymatrix{
& CH_0(D)^{\wedge } \ar[r] \ar[d] & \bigl(CH_0(Y_0)^{\wedge}\bigr)^{ \oplus 2} \ar[r] \ar[d] & CH_0(Y)^{\wedge } \ar[d] \ar[r] &0\\
0 \ar[r] & K_0(k)^{\wedge } \ar[r] & \bigl( K_0(k)^{\wedge}\bigr)^{ \oplus 2} \ar[r] & K_0(k)^{\wedge } \ar[r] & 0. 
}
\end{align*}
Applying the snake lemma to this diagram, we obtain an exact sequence
\begin{align*}
A_0(D)\longrightarrow A_0(Y_0)^{\oplus 2} \longrightarrow A_0(Y).
\end{align*}
Simiarly to (1), the sequence (2) is obtained by applying the snake lemma to the following commutative diagram with exact rows (cf. Sublemma \ref{L6}):
\begin{align*}
\xymatrix{
& \pi_1^{ab}(D) \ar[r] \ar@{->>}[d] & \pi_1^{ab}(Y_0)^{\oplus 2} \ar[r] \ar@{->>}[d] & \pi_1^{ab}(Y)\ar@{->>}[d]\\
0 \ar[r] & G_k^{ab} \ar[r] & {G_k^{ab}}^{\oplus 2} \ar[r] & G_k^{ab} \ar[r] & 0. 
}
\end{align*}
\end{proof}\bigskip

Now we define the map 
\begin{align*}
\delta _Y^{geo} : H_1(\Gamma _D, \hat{\mb{Z}})_{\overline{D}} \longrightarrow CH_0(Y)
\end{align*}
to be that induced by the following commutative diagram with exact rows (cf. Lemma \ref{L7}, Proposition \ref{P1}):
\begin{align*}
\xymatrix{
A_0(D) \ar[r] \ar[d] & A_0(Y_0)^{\oplus 2} \ar[r] \ar[d]^{\simeq } & A_0(Y) \ar[d]\\
\pi_1^{geo}(D) \ar[r] \ar@{->>}[d] & \pi_1^{geo}(Y_0)^{\oplus 2} \ar[r] & \pi_1^{geo}(Y)\\
H_1(\Gamma _D, \hat{\mb{Z}})_{\overline{D}}. 
}
\end{align*}
Here the bijectivity of the middle vertical map is due to Kato$-$Saito \cite{KS}. The following proposition plays a key role in the proof of the main result.
\begin{prop}\label{P4}
The following diagram commutes:
\begin{align}\label{p41}
\xymatrix{
H_1(\Gamma _{\overline{D}}, \mb{Z}) \ar[r] \ar[d]^{\sigma } & H_1(\Gamma _D, \hat{\mb{Z}})_{\overline{D}} \ar[d]^{\delta _Y^{geo}}\\
H_1(\Gamma _D, \mb{Z}) \ar[r]^{\delta _Y} & CH_0(Y).
}
\end{align}
\end{prop}
\begin{proof}
The commutativity of this diagram follows from the constructions of $\delta _Y, \delta _Y^{geo}$ and the commutativity of the following diagrams (cf. Lemma \ref{L2}, Proposition \ref{P1})
\begin{align*}
\xymatrix{
\pi_1^{geo}(D) \ar[r] \ar[d] & \pi_1^{ab}(D)\ar[d] & A_0(Y) \ar[d] \ar[r] &CH_0(Y)\ar[d]\\
H_1(\Gamma _D, \hat{\mb{Z}})_{\overline{D}} \ar[r] & H_1(\Gamma _D, \hat{\mb{Z}}), & \pi_1^{geo}(Y)\ar[r] &\pi_1^{ab}(Y), \\
\pi_1^{ab}(\overline{D}) \ar[r] \ar[d] & \pi_1^{ab}(D)\ar[d] \\
H_1(\Gamma _{\overline{D}}, \hat{\mb{Z}}) \ar[r]^{\hat{\sigma }} & H_1(\Gamma _D, \hat{\mb{Z}}). 
}
\end{align*}
\end{proof}

We use the following map in the proof of the main result.
\begin{lem}\label{L8}
There exists an injective homomorphism
\begin{align*}
\xymatrix{
\psi : \mathrm{Coker}\big( A_0(D)\rightarrow A_0(Y_0) \bigr) \ar@{^{(}->}[r] &A_0(Y).
}
\end{align*}
\end{lem}

\begin{proof}
We consider the following commutative diagram (cf. Lemma \ref{L7})
\begin{align*}
\xymatrix{
A_0(D)\ar[r] \ar@2{-}[d] &A_0(Y_0) \ar[d]^{\xi }\\
A_0(D) \ar[r] &A_0(Y_0)^{\oplus 2} \ar[r] & A_0(Y),
}
\end{align*}
where $\xi $ maps an element $a$ of $A_0(Y_0)$ to an element $(a, -a)$ of $A_0(Y_0)^{\oplus 2}$. From this diagram, we obtain an injective homomorphism
\begin{align*}
\xymatrix{
\psi : \mathrm{Coker}\big( A_0(D)\rightarrow A_0(Y_0) \bigr) \ar@{^{(}->}[r] &A_0(Y). 
}
\end{align*}
\end{proof}
%%%Proof of Theorem%%%%%%%%%%%%%%%%%%%%%%%%%%%%%%%%%%%%%%%%%%%%%%%%%%%%%%%%%%%%%%%%%%%%%%%%%%%%%%%%%%%%%%%%%%%%%%%%%%%%%%%%%%%%%%%%%%%%%%%%%%%%%%%%%%%%%
%%%%%%%%%%%%%%%%%%%%%%%%%%%%%%%%%%%%%%%%%%%%%%%%%%%%%%%%%%%%%%%%%%%%%%%%%%%%%%%%%%%%%%%%%%%%%%%%%%%%%%%%%%%%%%%%%%%%%%%%%%%%%%%%%%%%%%%%%%%%%%%%%%%%
\subsection{Proof of Theorem \ref{IT}}

Let $\ell $ be a prime number. We write $\Theta _{\ell }$ for the $G_k$-module
\begin{align*}
\mathrm{Coker}\bigl( \pi_1^{ab}({\overline{D}}^{(1)})^{pro-\ell }\longrightarrow \pi_1^{ab}(\overline{Y_0})^{pro-\ell }\bigr).
\end{align*}
We consider the following $G_k$-equivariant homomorphism
\begin{align*}
\alpha ^{(\ell) } : H_1(\Gamma _{\overline{D}},\mb{Z}_{\ell }) \longrightarrow \Theta _{\ell }
\end{align*}
induced by the following commutative diagram with exact rows
\begin{align*}
\xymatrix{
&\pi_1^{ab}({\overline{D}}^{(1)})^{pro-\ell }\ar[r] \ar[d] & \pi_1^{ab}(\overline{D})^{pro-\ell } \ar[r] \ar[d] &H_1(\Gamma _{\overline{D}}, \mb{Z}_{\ell }) \ar[r] \ar[d]&0\\
0\ar[r] & \pi_1^{ab}(\overline{Y_0})^{pro-\ell } \ar[r]^{id} & \pi_1^{ab}(\overline{Y_0})^{pro-\ell } \ar[r] &0.
}
\end{align*}

By the weight argument, Matsumi, Sato and Asakura \cite[Thm. 3.3]{MSA} proved the following:
\begin{lem}[Matsumi$-$Sato$-$Asakura]\label{MSAT}
Let $\ell $ be an arbitrary prime number.\\
\textup{(1)} The image of $\alpha ^{(\ell )}$ is contained in $\bigl(\Theta _{\ell }\bigr)_{tors}$.\\
\textup{(2)} Assume that $G_k$ acts on $\bigl(\Theta _{\ell }\bigr)_{tors}$ trivially. Then the composite of canonical maps
\begin{align*}
\xymatrix{
\big(\Theta _{\ell }\bigr)_{tors} \ar[r]^{f_1}&\bigl(\big(\Theta _{\ell }\bigr)_{tors}\bigr)_{G_k} \ar[r]^{f_2}& \big(\Theta _{\ell }\bigr)_{G_k}
}
\end{align*}
is injective.
\end{lem}

\begin{thm}[Theorem \ref{IT}]\label{MT}
Let $\ell $ be an arbitrary prime number.\\
\textup{(1)} The $\ell $-primary part $G(Y)\{\ell \}$ of $G(Y)$ is a subquotient of $\big(\Theta _{\ell }\bigr)_{tors}$.\\
\textup{(2)} Assume that

\ \ \ \textup{(i)} each connected components of $Y^{(2)}$ has $k$-rational point,

\ \ \ \textup{(ii)} $G_k$ acts on $\big(\Theta _{\ell }\bigr)_{tors}$ trivially. 

Then $G(Y)\{\ell \}$ is isomorphic to the image of the map $\alpha ^{(\ell )}$.
\end{thm}

\begin{proof}
(1) For a finite abelian group $M$, we write $M^{(\ell )}$ for the maximal $\ell $-quotient. Since $A_0(Y)$ is finite (cf. Lemma \ref{L1}(1)), the $\ell $-primary part $G(Y)\{\ell \}$ is identified with $G(Y)^{(\ell )}$, and hence identified with the image of the composite map
\begin{align*}
\big(\delta _Y\circ \sigma \bigr)^{(\ell )} : H_1(\Gamma _{\overline{D}},\mb{Z})\longrightarrow A_0(Y)\longrightarrow A_0(Y)^{(\ell )}.
\end{align*}
From the commutativity of the diagram \eqref{p41} in Proposition \ref{P4} and the constructions of $\delta _Y^{geo}$ and $\alpha ^{(\ell )}$, the map $\big(\delta _Y\circ \sigma \bigr)^{(\ell )}$ is decomposed as follows:
\begin{align*}
\xymatrix{
H_1(\Gamma _{\overline{D}}, \mb{Z}) \ar[r]^{ \ \ \ \ \ \ \alpha ^{(\ell )}} & \Theta _{\ell } \ar[r] & \big(\Theta _{\ell }\bigr)_{G_k} \ar[r]^{\eta ^{(\ell )}} & A_0(Y)^{(\ell )},
}
\end{align*}
where $\eta ^{(\ell )}$ denotes the following composite map:
\begin{align*}
\big(\Theta _{\ell }\bigr)_{G_k} &\simeq \mathrm{Coker}\bigl( \pi_1^{geo}(D^{(1)})^{pro-\ell }\rightarrow \pi_1^{geo}(Y_0)^{pro-\ell }\bigr)\\
&\simeq \mathrm{Coker}\bigl( A_0(D^{(1)})^{(\ell) }\rightarrow A_0(Y_0)^{(\ell )}\bigr)\\
&\rightarrow \mathrm{Coker}\bigl( A_0(D)^{(\ell )}\rightarrow A_0(Y_0)^{(\ell )}\bigr) \stackrel{\psi ^{(\ell )}}{\longrightarrow }A_0(Y)^{(\ell )}.
\end{align*}
From Lemma \ref{MSAT}(1), the image of $\alpha ^{(\ell) }$ is contained in $\big(\Theta _{\ell }\bigr)_{tors}$. Thus, $G(Y)\{\ell \}$ is a subquotient of $\big(\Theta _{\ell }\bigr)_{tors}$.

(2) It suffices to show that the composite of canonical maps
\begin{align*}
\xymatrix{
\mathrm{Im}\bigl( \alpha ^{(\ell )}\bigr) \ar@{^{(}->}[r]& \big(\Theta _{\ell }\bigr)_{tors} \ar[r]^{f_1}&\bigl(\big(\Theta _{\ell }\bigr)_{tors}\bigr)_{G_k} \ar[r]^{f_2}& \big(\Theta _{\ell }\bigr)_{G_k} \ar[r]^{\eta ^{(\ell) }}& A_0(Y)^{( \ell) }
}
\end{align*}
is injective under the assumptions. Here the first map is injective by (1). From Lemma \ref{MSAT}(2), the composite map $f_2\circ f_1$ is injective. Under the assumption (i), $\eta ^{(\ell)}$ coincides with the map $\psi ^{(\ell )}$ from Lemma \ref{L1}(2), and hence is injective from Lemma \ref{L8}.
\end{proof}
%%%Examples%%%%%%%%%%%%%%%%%%%%%%%%%%%%%%%%%%%%%%%%%%%%%%%%%%%%%%%%%%%%%%%%%%%%%%%%%%%%%%%%%%%%%%%%%%%%%%%%%%%%%%%%%%%%%%%%%%%%%%%%%%%%%%%%%%%%%%%%%
%%%%%%%%%%%%%%%%%%%%%%%%%%%%%%%%%%%%%%%%%%%%%%%%%%%%%%%%%%%%%%%%%%%%%%%%%%%%%%%%%%%%%%%%%%%%%%%%%%%%%%%%%%%%%%%%%%%%%%%%%%%%%%%%%%%%%%%%%%%%%%%%%%%%
\section{Examples}

We here give two examples of simple normal crossing surfaces over $k$.  One is a surface $Y_1$ for which $H_2(\Gamma _{Y_1\otimes F},\mb{Z})$ is not trivial but $\rho _{Y_1\otimes F}$ is injective for any finite extension $F/k$. Another is a surface $Y_2$ for which $\rho _{Y_2\otimes F}$ is not injective for any finite extension $F/k$. We also see that the map $\rho_{Y_2\otimes F}/n$ is not injective for some positive integer $n$. On the other hand, there is a simple normal crossing surface $Y_3$ for which $\rho _{Y_3}$ is not injective but $\rho _{Y_3\otimes E}$ is injective for any sufficiently large finite extension $E/k$ (cf. \cite{Sat}).
\begin{ex}\label{E2}
Let $Y_0:=\mb{P}^1\times _k\mb{P}^1$ and $D$ be the simple normal crossing divisor on $Y_0$ defining the following polynomial :
\begin{align*}
D \ : \ (x^2-y^2)(z^2-w^2)=0 \subset Y_0.
\end{align*}
Then, for any finite extension $F/k$, we have
\begin{align*}
H_1(\Gamma _{D\otimes F},\mb{Z})\simeq \mb{Z}.
\end{align*}

On the other hand, since $\pi _1^{ab}(\overline{Y_0})=0$, we have $\Theta =0$. Hence, from Theorem \ref{MT}, for the following simple normal crossing surface $Y_1:=\bigl( Y_0\times _kO \bigr) \cup \bigl( Y_0\times _k{\infty }\bigr) \cup \bigl( D\times _k\mb{P}^1\bigr)$, the map $\rho _{Y_1}$ is injective for any finite scalar extension.
\end{ex}

\begin{ex}[cf. \cite{MSA}]\label{E1}
Let $n>1$ be a natural number and $(n,6\cdot ch(k))=1$. We assume $k$ is a finite field containing a primitive $n$-th root of unity $\zeta $. We consider a Fermat surface
\begin{align*}
V  :  T_0^n+T_1^n+T_2^n+T_3^n=0 \ \subset \mb{P}_k^3,
\end{align*}
and a free action on $V$
\begin{align*}
\tau : \bigl( T_0 : T_1 : T_2 : T_3 \bigr) \longmapsto \bigl( T_0 : \zeta T_1 : \zeta ^2T_2 : \zeta ^3T_3 \bigr).
\end{align*}
Then we have a projective smooth surface $Y_0:=V/<\tau>$.

Now we consider $2n$ lines on $V$ : $j=1,\dots ,n-1$
\begin{align*}
L _1& : T_0+T_1=T_2+T_3=0,\\
L _2& : T_0+T_1=T_2+\zeta T_3=0 ,\\
{L _1}^{{\tau }^j}& : T_0+{\zeta }^jT_1=T_2+{\zeta }^jT_3=0 ,\\
{L _2}^{{\tau}^j}& : T_0+{\zeta }^jT_1=T_2+{\zeta }^{j+1}T_3=0 .
\end{align*}
Then the following divisor $L$ on $V$ is a connected simple normal crossing divisor:
\begin{align*}
L = L _1\cup L _2\cup {L _1}^{{\tau }}\cup \dots \cup {L _1}^{{\tau }^{n-1}}\cup {L _2}^{{\tau }^{n-1}}.
\end{align*}
This divisor $L $ is stable under the action of $<\tau >$.

Let $\varphi :V\longrightarrow Y_0$ and $C_i=\varphi _{\ast }(L _i)$ $(i=1,2)$. Since $C_i$ is  isomorphic to $L_i$, $C_i$ is a nonsingular rational curve on $Y_0$ and $D=C_1\cup C_2$ is a simple normal crossing divisor on $Y_0$. Moreover every singular points of $D$ are $k$-rational. 

Since $\overline{V}$ is a hypersurface in $\mb{P}^3_{k^{sep}}$, $\pi _1^{ab}(\overline{V})=0$ (cf. \cite[Lemma 3.5.]{Sat}). Hence we have
\begin{align*}
\pi _1^{ab}(\overline{Y_0})\simeq <\tau >\simeq \mb{Z}/n.
\end{align*}
Since $C_i$is rational curves, $\pi _1^{ab}(\overline{C_i})=0$. Therefore, $G_k$ acts on $\Theta _{tors}$ trivially.

On the other hand, because $\varphi $ induces the completely splitting covering $L \longrightarrow D$, 
\begin{align*}
\alpha : H_1(\Gamma _{\overline{D}},\mb{Z})\longrightarrow \pi _1^{ab}(\overline{Y_0})
\end{align*}
is surjective.

We then put $Y_2 := \bigl( Y_0\times _kO \bigr) \cup \bigl( Y_0\times _k{\infty }\bigr) \cup \bigl( D\times _k\mb{P}^1\bigr)$. From Theorem \ref{MT}, $\mathrm{Ker}(\rho _{Y_2})\simeq \mb{Z}/n$. Thus $\rho _{Y_2}$ is not injective. Moreover, we have $\mathrm{Ker}(\rho _{{Y_2}\otimes F})\simeq \mb{Z}/n$ for any finite extension $F/k$, therefore the map $\rho_{{Y_2}\otimes F}$ is not injective. We also know that the map $\rho _{{Y_2}\otimes F}/n$ is not injective.
\end{ex}
\begin{rmk}
We can construct an example of a higher dimensional variety for which the reciprocity map is not injective for any finite scalar extension as follows; Let $X$ be a projective smooth and geometrically irreducible variety over $k$. We consider the fiber product $Y_2\times _kX$, where $Y_2$ is the surface of Example \ref{E1}. Then $\rho_{Y_2\times _kX}$ is not injective for any finite scalar extension. This follows from the following commutative diagram:
\begin{align*}
\xymatrix{
H_2(\Gamma _{Y_2\times X},\mb{Z}) \ar@2{-}[d] \ar[r]^{\delta _{Y_2\times X}} &CH_0(Y_2\times X) \ar[d] \ar[r]^{\rho_{Y_2\times X}} &\pi_1^{ab}(Y_2\times X) \ar[d]\\
H_2(\Gamma _{Y_2},\mb{Z}) \ar[r]^{\delta _{Y_2}} &CH_0(Y_2) \ar[r]^{\rho_{Y_2}} &\pi_1^{ab}(Y_2).
}
\end{align*}
\end{rmk}
%%%Acknowledgements%%%%%%%%%%%%%%%%%%%%%%%%%%%%%%%%%%%%%%%%%%%%%%%%%%%%%%%%%%%%%%%%%%%%%%%%%%%%%%%%%%%%%%%%%%%%%%%%%%%%%%%%%%%%%%%%%%%%%%%%%%%%%%%%%%%%%
%%%%%%%%%%%%%%%%%%%%%%%%%%%%%%%%%%%%%%%%%%%%%%%%%%%%%%%%%%%%%%%%%%%%%%%%%%%%%%%%%%%%%%%%%%%%%%%%%%%%%%%%%%%%%%%%%%%%%%%%%%%%%%%%%%%%%%%%%%%%%%%%%%%%
\section*{Acknowledgements}
The author expresses his sincere gratitude to K. Sato for his valuable comments and discussions. 

%%%References%%%%%%%%%%%%%%%%%%%%%%%%%%%%%%%%%%%%%%%%%%%%%%%%%%%%%%%%%%%%%%%%%%%%%%%%%%%%%%%%%%%%%%%%%%%%%%%%%%%%%%%%%%%%%%%%%%%%%%%%%%%%%%%%%%%%%
%%%%%%%%%%%%%%%%%%%%%%%%%%%%%%%%%%%%%%%%%%%%%%%%%%%%%%%%%%%%%%%%%%%%%%%%%%%%%%%%%%%%%%%%%%%%%%%%%%%%%%%%%%%%%%%%%%%%%%%%%%%%%%%%%%%%%%%%%%%%%%%%%%%%

\end{document}